\documentclass[11pt,reqno]{amsart}
\usepackage{amsmath, amssymb, amsfonts, latexsym, mathrsfs, bm, mathtools}
\usepackage[all]{xy}
\usepackage[pdftex]{graphicx}
\usepackage[OT2,T1]{fontenc}
\usepackage{times}
\usepackage[english]{babel}
\usepackage{tikz-cd}
\usepackage{hyperref} 
\usepackage{color}
\newcommand{\rd}{\color{black}}
\newcommand{\bl}{\color{black}}
\newcommand{\gl}{\color{black}}
\newcommand{\tl}{\color{black}}
\newcommand{\al}{\color{black}}
\newcommand{\dl}{\color{black}}
\newcommand{\cl}{\color{black}}

\newcommand{\Z}{\mathbb{Z}}
\newcommand{\Q}{\mathbb{Q}}

\newcommand{\F}{\mathbb{F}}

\newcommand{\lrta}{\longrightarrow}

\newcommand{\h}{\mathcal{H}}

\newcommand{\sel}{\mathrm{Sel}}
\newcommand{\gal}{\mathrm{Gal}}

\newcommand{\ilim}{\varinjlim}
\newcommand{\plim}{\varprojlim}
\newcommand{\hone}{\mathrm{H}^1}
\newcommand{\hzero}{\mathrm{H}^0}

\renewcommand{\(}{\left\(}
\renewcommand{\)}{\right\)}

\numberwithin{equation}{section}
\theoremstyle{plain}
\newtheorem{theorem}{Theorem}[section]
\newtheorem{lemma}[theorem]{Lemma}
\newtheorem{corollary}[theorem]{Corollary}

\theoremstyle{definition}
\newtheorem{definition}[theorem]{Definition}
\newtheorem{remark}[theorem]{Remark}
\newtheorem{assumption}[theorem]{Assumption}

\newenvironment{nouppercase}{%
\renewcommand{\uppercasenonmath}[1]{}}{} 
 
\newcommand{\Keywords}[1]{\par\noindent
{\small{Keywords and phrases}: #1}}

\newcommand{\AMS}[1]{\par\noindent
{\small{AMS Subject Classification}: #1}}
 
\author{Suman Ahmed, Chandrakant Aribam, Sudhanshu Shekhar}
\address{Indian Institute of Science Education and Research (IISER) Mohali, Knowledge City, Sector 81, Manauli, SAS Nagar, Punjab, 140306, India.}
\email{sahmed@iisermohali.ac.in}
\address{Indian Institute of Science Education and Research (IISER) Mohali, Knowledge City, Sector 81, Manauli, SAS Nagar, Punjab, 140306, India.}
\email{aribam@iisermohali.ac.in}
\address{Indian Institute of Technology (IIT) Kanpur, Kanpur, Uttar Pradesh, 208016, India.}
\email{sshekhars2012@gmail.com}

\begin{document}
 
\title{Root numbers and parity of local Iwasawa invariants}
   

\begin{abstract}
Given two elliptic curves $E_1$ and $E_2$ defined over the field of rational numbers, $\Q$, with good reduction at an odd prime $p$ and equivalent mod $p$ Galois representation, we compare the $p$-Selmer rank, global and local root numbers of $E_1$ and $E_2$ over number fields. 
\end{abstract}

\begin{nouppercase}
\maketitle
\end{nouppercase}
\let\thefootnote\relax\footnotetext{
\AMS{11G05, 11G07, 11R23, 14H52.}
}

\section{Introduction} \label{s1}
Let $F$ be a number field. Fix an odd prime $p$ and let $\overline{F}$ denote a fixed algebraic closure of $F$. For an algebraic field extension $L/F$ and an elliptic curve $E$ defined over $F$, the $p$-Selmer group $\sel_p(E/L)$ of $E$ over $L$ is defined by the exact sequence $$ 0 \lrta \sel_p(E/L) \lrta \hone(L,E[p^\infty]) \stackrel{\oplus_w\delta_w}{\lrta} \prod_{w} \hone(L_w,E(\overline{L_w})) $$
where $w$ varies over the set of primes of $L$, $\overline{L_w}$ denotes a fixed algebraic closure of $L_w$, $E(\overline{L_w})$ denotes the $\overline{L_w}$ points of $E$, $E[p^\infty]$ denotes the $p$-power torsion points of $E$ defined over $\overline{F}$ and $\oplus\delta_w$ denotes the natural restriction map.

Let $E_1$ and $E_2$ be two elliptic curves defined over the field of rational numbers, $\Q$, with good  reduction at $p$. 
 For an abelian group $M$, let $M[p]$ denote the $p$-torsion subgroup of $M$. 
We say that $E_1$ and $E_2$ are congruent at $p$ if  $E_1[p]\cong E_2[p]$ as modules over $\text{Gal}(\overline{\Q}/\Q)$. 
The aim of this article is to study the variation of the parity of the $p$-Selmer ranks, global and local root numbers of $E_1$ and $E_2$ over $F$. 

Let $N_i$ denote the conductor of $E_i$ over $F$, and $\overline{N_i}$ be the prime to $p$-{\rd part of the} conductor of the Galois {\rd representation}  $E_i[p]$ over $F$  for $i=1,2$. 
Let $\Sigma$ be the finite set of primes of $F$ containing the primes of bad reduction of $E_1$ and $E_2$, the infinite primes and the primes above $p$. 
Put $$ \Sigma_0:= \{v \in \Sigma  \mid  v \mid N_1/\overline{N}_1 \ \text{or} \ v \mid N_2/\overline{N}_2  \}. $$
Let $S_{E_i}$ be the set of primes  $v\in \Sigma_0$ such that $E_i$ has split multiplicative reduction at $v$, and $s_p(E_i/F)$ denote the $\Z_p$-rank of the Pontryagin dual of $\sel_p(E_i/F)$. 
{ \cl We denote the set of $p$-th roots of unity by $\mu_p$.
Let $ T $ be the set consisting of primes $v$ of $F$ in $\Sigma_0$ such that $\mu_p\not\subset F_v$ and one of the elliptic curves has multiplicative reduction while the other has additive reduction at $v$. }
Note that $T$ can be non-empty only when $p=3$ (see  Lemma \ref{l2.10}).
In this article, we prove the  following.  

\begin{theorem}\label{t1.1}

Let $E_1$ and $E_2$ be two elliptic curves defined over $\Q$ with good and ordinary reduction at $p$.
We further assume the following. \\
(a) $E_1[p]$ is an irreducible $\gal(\overline{F}/F)$-module.\\
(b) $\sel_p(E_1/F_{cyc})[p]$ is finite where $F_{cyc}$ denotes the cyclotomic $\Z_p$-extension of $F$.\\
{\rd Let} $E_1$ and $E_2$ are congruent at $p$.
{\rd Then} $$
s_p(E_1/F) + \mid S_{E_1} \mid \equiv s_p(E_2/F) + \mid S_{E_2} \mid + \mid T \mid  \,  \text{mod} \enspace 2 . $$

\end{theorem}

The above theorem is proved in \cite[Theorem 1.1]{Sh} under the assumption that $ \mu_p \subset F. $ 
Next assume that $E_1$ and $E_2$ have good supersingular reduction at $ p. $
In this case $\sel_p(E_i/F_{cyc})[p]$ are not finite groups so we consider the $\pm$ Selmer groups $\sel^{\pm}_p(E_i/F_{cyc})$ associated to $E_i, i=1,2$ as defined by Kobayashi (see Section \ref{s4}). 
Then we prove the following {\cl analogue} of Theorem \ref{t1.1} for elliptic curves with supersingular reduction at $ p $.
\begin{theorem}\label{t1.2}
Let $E_1$ and $E_2$ be two elliptic curves defined over $\Q$ with good and supersingular reduction at the prime $p$.
We further assume the following. \\
(a) $E_1[p]$ is an irreducible $\gal(\overline{F}/F)$-module.\\
(b) $\sel_p^{-}(E_1/F_{cyc})[p]$ is finite.\\
(c) $p$ splits completely over $F$. \\
If $E_1$ and $E_2$ are congruent at $p$
then $$ s_p(E_1/F) + \mid S_{E_1} \mid \equiv s_p(E_2/F) + \mid S_{E_2} \mid + \mid T \mid  \,  \text{mod} \enspace 2 . $$

\end{theorem}

In this article, we also consider an analogue of the above theorem in terms of root numbers of elliptic curves. 
For an elliptic curve $E$ and a prime $v$ of $F$, let $w(E/F_v)$ denote the local root number of $E$ over $F_v$ and $w(E/F)$ denote the global root number of $E$ over $F$ (see Section \ref{s5}). 

\begin{theorem}\label{t1.3}
Let $E_1$ and $E_2$ be two elliptic curves defined over $\Q$ with good reduction at $p$ {\rd and $E_1[p]\cong E_2[p]$. }
Then  $$  \dfrac{w(E_1/F)}{w(E_2/F)} =  (-1)^{\mid S_{E_1} \mid - \mid S_{E_2} \mid + \mid T \mid }  $$
\end{theorem}

In this context, we also recall the {\it $p$-parity conjecture} which states that for every elliptic curve $E$ defined over $F$,  $$ w(E/F)= (-1)^{s_p(E/F) }.$$
{\gl This conjecture is known over $ \Q $ (\cite{Mo}, \cite{Ne}, \cite{Ki}, \cite{DD}) and for `most' elliptic curves over totally real fields (\cite{Ne2} for odd $ p $ and \cite{Bi} for $ p=2 $). }
{\al In particular, the $p$-parity conjecture together with Theorem \ref{t1.3} imply Theorem \ref{t1.1} and Theorem \ref{t1.2} irrespective of the assumptions. 
However we give independent proofs of Theorem \ref{t1.1}, Theorem \ref{t1.2} and Theorem \ref{t1.3} in this article.}

For a finite prime $v\nmid p$ of $F$ and a prime $w \mid v$ of $F_{cyc}$, let $\sigma_E^w$ denote the  $\Z_p$-rank of the Pontryagin dual of  $\hone(F_{cyc,w},E[p^\infty])$. 
{\rd Note} that the Pontryagin dual of $\hone(F_{cyc,w},{\cl E[p^\infty]})$ is a finitely generated $\Z_p$-module.
Put $\h_v(E/F_v)=\prod_{w|v} \hone(F_{cyc,w},{\cl E[p^\infty]})$ where $w$ varies over the set of primes of $F_{cyc}$ dividing $v$. {\rd Let $\Gamma = \text{Gal}(F_{cyc}/F) \cong \Z_p$ and} $\Lambda := \plim_n\Z_p[\Gamma/\Gamma^{p^n}] $ denote the Iwasawa algebra over $\Gamma$ where the inverse limit is taken with respect to the natural projection maps. 

Since there exists finitely many primes of $F_{cyc}$ dividing $v$, the Pontryagin dual  $\widehat{\h_v(E/F_v)}$ of $\h_v(E/F_v)$ is a finitely generated $\Z_p$-module. 
In particular,  $\widehat{\h_v(E/F_v)}$ is a torsion $\Lambda$-module.  
{\rd Let $\tau_{E}^v=\text{rank}_{\Z_p}\widehat{\h_v(E/F_v)}.$
Then $\tau_{E}^v$ is the Iwasawa $\lambda$-invariant of $\widehat{\h_v(E/F_v)}$.  }

\begin{theorem} \label{t1.4}
Let $E_1$ and $E_2$ be two elliptic curves defined over $\Q$ with good reduction at $p$ {\rd and $E_1[p]\cong E_2[p]$.}
Let $v$ be a finite prime of $F$ {\rd with $v\nmid p$}. 
Then $$  \dfrac{w(E_1/F_v)}{w(E_2/F_v)} = (-1)^{{\tau_{E_1}^v}-{\tau_{E_2}^v} } $$
\end{theorem}

{\rd Let $T_p(E) $ denote the Tate module of $ E $ and $V_p(E) := T_p(E) \otimes \Q_p$. 
Let $v$ be a finite prime of $F$ with $v\nmid p$.  
The natural inclusion $T_p(E)\lrta V_p(E)$ induces a homomorphism $f : \hone(F_v,T_p(E))\lrta \hone(F_v,V_p(E))$. 
Next under the homomorphism $\hone(F_v,T_p(E))\lrta \hone(F_v,E[p])$, we denote $\mathcal{F}(E):= \text{Image}(\text{ker}(f) \lrta \hone(F_v,E[p]))$.  
Let $\mathcal{F}_i:=\mathcal{F}(E_i), i=1,2$ and }$$ d_v(\mathcal{F}_1,\mathcal{F}_2):= \text{dim}_{\Z/p} \mathcal{F}_1/(\mathcal{F}_1\cap \mathcal{F}_2) \enspace  \text{mod} \enspace  2.  $$ 
The invariant $d_v(\mathcal{F}_1,\mathcal{F}_2)\in \Z/2$ is the {\it Arithmetic local constant} of Mazur and Rubin \cite{MR}. 
This invariant has been extensively studied by Nekovar \cite{Ne3}. 
{\dl Assuming that $ E_1[p] $ and $ E_2[p] $ are symplectically isomophic $ G_{\Q} $-modules, he proved
$$ \dfrac{w(E_1/F_v)}{w(E_2/F_v)}=  d_v(\mathcal{F}_1,\mathcal{F}_2) $$
for every finite prime $v\nmid p$ of $F$ \cite[Theorem 2.17]{Ne3}.
Now under similar assumptions, the above equality together with Theorem \ref{t1.4} imply $$ d_v(\mathcal{F}_1,\mathcal{F}_2) =  (-1)^{{\tau_{E_1}^v}-{\tau_{E_2}^v} }. $$ }
Hence Theorem \ref{t1.4} can be considered as an Iwasawa theoretic analogue of the result of Nekovar. 

In Section \ref{s2}, we compare the parity of local Iwasawa invariants of $ E_1 $ and $ E_2 $ congruent at $ p $.
Using the parity of local Iwasawa invariants, we prove Theorem \ref{t1.1} in Section \ref{s3} and Theorem \ref{t1.2} in Section \ref{s4}. 
Finally, in Section \ref{s5} we compare the root numbers of $ E_1 $ and $ E_2 $. 
{\rd The techniques are a careful and detailed analysis of the local cohomology groups.
More precisely, we compare the local Iwasawa $\lambda$-invariants $\tau_{E_1}^v$, $\tau_{E_2}^v$ and also the local root numbers of $E_1$ and $E_2$ under congruence.
The idea of checking for variation of certain invariants attached to elliptic curves is inspired by the work of Greenberg \& Vatsal \cite{GV}. }

{\cl The authors gratefully acknowledge the support of SERB and DST, Government of India. 
We would also like to thank the referee for careful reading and useful comments which helped us improve the exposition.
This work was written up when the first author was a Postdoctoral fellow at IISER, Mohali and the third author was visiting IISER, Mohali.}

\section{Parity of local Iwasawa invariants} \label{s2}

Let $ E_1 $ and $ E_2 $ be two elliptic curves defined over $ \Q $ with good reduction at an odd prime $ p. $ 
Recall that $ E_1 $ and $ E_2 $ are congruent at the prime $ p $ if $ E_1[p] \cong E_2[p] $ as modules over $ \text{Gal}(\overline{\Q}/\Q). $

Let $ F $ be a number field and $ v \nmid p $ be a finite prime of $ F $.
Then for every prime $ w \mid v $ of $ F_{cyc} $ and an elliptic curve $ E $ defined over $ F, $ the Pontryagin dual of $ \hone(F_{cyc,w}, E[p^{\infty}]) $ is a finitely generated $ \Z_p$-module \cite[Proposition 2]{Gr}.
Let $ \sigma_{E}^{w} $ denote the $ \Z_p $-corank of $ \hone(F_{cyc,w}, E[p^{\infty}]) $.
{\bl In this section, we compare the parity of $ \sigma_{E_1}^{w} $ and $ \sigma_{E_2}^{w} $ of two congruent elliptic curves $ E_1 $ and $ E_2 $ with different reduction types at $ v $  in the series of lemmas given below.}

For an elliptic curve $ E $ over $ F $, consider the following exact sequence {\rd of Galois modules} $$ 0 \longrightarrow E[p] \longrightarrow E[p^\infty] \xrightarrow{p} E[p^\infty] \longrightarrow 0 $$ where the map $ p $ denotes multiplication by $ p $.
{\rd From the long exact sequence of Galois cohomology w.r.t.} $ G=\text{Gal}({\rd \overline{F}_{v}}/F_{cyc,w})
 $, we get the following exact sequence
\begin{equation}\label{e2.2}
0 \longrightarrow \dfrac{\hzero(G, E[p^{\infty}])}{p\hzero(G, E[p^{\infty}])} \longrightarrow \hone(G, E[p]) \longrightarrow \hone(G, E[p^{\infty}])[p] \longrightarrow 0
\end{equation} 

\begin{lemma}[\rd \bf Good \textendash \ Good]\label{l2.1}
{\rd Let $ v $ be a finite prime of $ F $ with $ v \nmid p $ and $ w $ be a prime of $ F_{cyc} $ with $ w \mid v $.}
Let $ E_1 $ and $ E_2 $  be two elliptic curves with good reduction at $ v $ and $ E_1[p] \cong E_2[p], $ then $ \sigma_{E_1}^{w} =  \sigma_{E_2}^{w}. $
\end{lemma}

\begin{proof}
It is well known that for an elliptic curve $ E $ over $ F $, $ \hone(G, E[p^{\infty}]) $ is divisible \cite[Lemma 4.5 and the successive paragraph]{Gr1}.
{\rd Therefore} $ \text{corank}_{\Z_p}\hone(G, E[p^{\infty}])=\text{dim}_{\F_p}\hone(G, E[p^{\infty}])[p] $.
Since $ E_i $ has good reduction at $ v $, by an argument similar to \cite[Lemma 4.1.2]{EPW} one can prove that $  \hzero(G, E_i[p^{\infty}]) $ is divisible for $ i=1,2 $.
{\rd Now it follows from the exact sequence (\ref{e2.2}) that} $ \hone(G, E_i[p]) \cong \hone(G, E_i[p^{\infty}])[p],$ $ i=1,2. $
{\rd Together with} $ E_1[p] \cong E_2[p] $, we have $ \sigma_{E_1}^{w} =  \sigma_{E_2}^{w}. $
\end{proof}

\begin{lemma}[\rd \bf Additive]\label{l2.2}
{\rd Let $ v $ be a finite prime of $ F $ with $ v \nmid p $ and $ w $ be a prime of $ F_{cyc} $ with $ w \mid v $.}
Let $ E $ be an elliptic curve with additive reduction at $ v $,  then $ \sigma_{E}^{w} = 0. $
\end{lemma}

\begin{proof}

{\cl Let $ F^\prime :=F_v(\mu_{p^{\infty}}) $ and $ H=\text{Gal}(\overline{F_v}/F^\prime) $ where $\mu_{p^{\infty}} $ denotes the group of all $ p $-power roots of unity.
Now from the inflation-restriction exact sequence, we have $ \hone(G, E[p^{\infty}]) \subseteq \hone(H, E[p^{\infty}]). $ }
Since {\rd $ F^\prime $} is unramified outside $ p, $ the reduction type does not change for $ E $ over {\rd $ F^\prime $} \cite[Chapter 7, Proposition 5.4]{Si}.
{ \cl Hence from \cite[Lemma 3.5]{Sh}, we get $ \hone(H, E[p^{\infty}])=0 $ which implies $ \hone(G,E[p^\infty])=0 $ i.e. $ \sigma_{E}^{w}=0. $ }

\end{proof}

\begin{lemma}[\rd \bf Good \textendash \ Additive]\label{l2.3}
{\rd Let $ v $ be a finite prime of $ F $ with $ v \nmid p$.}
Then it is not possible to have elliptic curves $ E_1 $ and $ E_2 $ such that $ E_1 $ has good reduction at $ v, E_2 $ has additive reduction at $ v $ and $ E_1[p] \cong E_2[p] $.
\end{lemma}

\begin{proof}

{ \rd Since $ E_1 $ has good reduction at $ v $ and $ E_1[p] \cong E_2[p] $, $ E_2[p] $ is unramified at $ v. $
Hence we have $ E_2[p] \subset E_2(F_v^{nr}) $ where $ F_v^{nr} $ denotes the maximal unramified extension of $ F_v. $ 
Since $ E_2[p]  $ has got finitely many points, there exists a finite extension $ L $ of $ F_v $ contained in $ F_v^{nr} $ such that $ E_2[p] \subset E_2(L) \subset E_2(F_v^{nr}). $
Let $ E:=E_2. $
Consider the following exact sequences
\begin{equation} \label{e2.11}
0 \longrightarrow E_k(F_v^{nr}) \longrightarrow E_0(F_v^{nr}) \longrightarrow \widetilde{E}_{ns}({\cl \overline{\F_v}}) \longrightarrow 0
\end{equation}
\begin{equation}
0 \longrightarrow E_0(L) \longrightarrow E(L) \longrightarrow E(L)/E_0(L) \longrightarrow 0
\end{equation}
where {\cl $ \F_v $ denotes the residue field at $ v $ and $ \widetilde{E}_{ns}(\overline{\F_v}) $} denotes the set of nonsingular points of $ \widetilde{E} $ over {\cl $ \overline{\F_v} $}, $ E_0  $ is the set of points with nonsingular reduction and $ E_k $ is the kernel of reduction.

By \cite[Theorem 7.6.1]{Si}, the group $ E(L)/E_0(L) $ has at most 4 elements. 
Since there are $ p^2 \geq 9 $ elements in $ E[p], $
there is a non-trivial element in $ E[p] $ whose image in $ E(L)/E_0(L) $ is trivial. 
Let $ P $ denote this element. 
Then $ P \in E_0(L). $ 
Therefore, the torsion element $ P \in E_0(F_v^{nr}). $

Since the group $ E_k(F_v^{nr}) $ is {\cl $ p $-torsion} free \cite[Proposition 7.3.1]{Si}, it follows from the exact sequence (\ref{e2.11}) that the image of $ P $
in $ \widetilde{E}_{ns}({\cl \overline{\F_v}}) $ is non-trivial. 
If $ E $ has additive reduction, then $ \widetilde{E}_{ns}({\cl \overline{\F_v}}) \cong {\cl \overline{\F_v}} $, as an additive group, which is also a {\cl $ p $-torsion} free group. 
We therefore arrive at a contradiction since the image of $ P $ is a non-trivial torsion element of $ \widetilde{E}_{ns}({\cl \overline{\F_v}}) $. 
It follows that $ E $ cannot have additive reduction at $ v. $ }

\end{proof}

Assume $E$ has multiplicative reduction at $p$.
{\cl Let} {\rd $ v $ be a finite prime of $ F $ with $ v \nmid p $ {\cl and $ G_v, I_v $ denote the decomposition and inertia groups of $ v $ over $ F $ respectively}.}
Let $\delta : G_v/I_v \longrightarrow \{\pm 1\} $ be the unique non-trivial unramified quadratic character of $ G_v $ if $ E $ has non-split multiplicative reduction, and let $\delta$ be the trivial character if $ E $ has split multiplicative reduction. 
{\cl Let $ \rho_{E,p}: \text{Gal}(\overline{\Q}/F) \longrightarrow \text{GL}_2(\Z_p) $ be the Galois representation attached to the Tate module of $ E $ over $ F $ and $ \overline{\rho}_{E,p} $ be the reduction modulo $ p $ of $ \rho_{E,p}. $  }
Then from \cite[Proposition 2.12]{DDT}, we have
$$  \rho_{E,p}\mid_{G_v}   \enspace \sim
     \left( {\begin{array}{cc}
      {\cl \omega_p} & * \\
      0 & 1 \\
     \end{array} } \right) \otimes \delta $$ where $ {\cl \omega_p} $ is the $ p $-adic cyclotomic character and 
$$ \overline{\rho}_{E,p}\mid_{G_v}   \enspace \sim
     \left( {\begin{array}{cc}
      \epsilon_p & \psi \\
      0 & 1 \\
     \end{array} } \right) \otimes \delta $$
where {\rd the mod $ p $ cyclotomic character is denoted by $ \epsilon_p $} {\cl and $ \psi \in \hone(G_{\Q_{l}}, \Z/p\Z \, (\epsilon_p)) \cong \Q_{l}^{\times}/(\Q_{l}^{\times})^{p} $ corresponds to the Tate period $ q_{E,l} \in \Q_{l}^{\times}/(\Q_{l}^{\times})^{p} $ for $ v \mid l $.}

{\cl Let $ \chi_p $ be the Teichmuller character and we identify it with $ \epsilon_p $ via the natural embedding $ \F_p^{*} \subset \Z_p^{*}. $ 
Then we have $ \omega_p \mid_{G} \,=\epsilon_p $ which implies}
 \begin{equation}\label{e2.3}
     \rho_{E,p}\mid_{G}   \enspace \sim
     \left( {\begin{array}{cc}
      \epsilon_p & * \\
      0 & 1 \\
     \end{array} } \right) \otimes \delta
 \end{equation}
 and
\begin{equation}\label{e2.4}
 \overline{\rho}_{E,p}\mid_{G}   \enspace \sim
     \left( {\begin{array}{cc}
      \epsilon_p & \psi \\
      0 & 1 \\
     \end{array} } \right) \otimes \delta
\end{equation}

\begin{remark}
Since { \rd $ F_{\text{cyc},w} $ is an unramified $ \Z_p $-extension of $ F_v $ for $ v \nmid p $ and $ w \mid v $ }, $ \epsilon_p \mid_G $  and $ \delta \mid_G $ are non-trivial iff $ \epsilon_p \mid_{G_v} $  and $ \delta \mid_{G_v} $ are non-trivial.
\end{remark}

\begin{lemma}[\rd \bf Multiplicative \textendash \ Additive]\label{l2.10}
Let $ p \geq 5 $ and {\rd $ v $ be a finite prime of $ F $ with $ v \nmid p $. }
Then it is not possible to have elliptic curves $ E_1 $ and $ E_2 $ such that $ E_1 $ has multiplicative reduction at $ v, E_2 $ has additive reduction at $ v $ and $ E_1[p] \cong E_2[p] $.
\end{lemma}

\begin{proof}

First, {\rd suppose that} $ E_1 $ has split multiplicative reduction at $ v $.
Let $ F^\prime:={\rd F_v(\mu_p)} $ and $ G^{\prime\prime}=\text{Gal}({\rd \overline{F}_{v}}/F^\prime_{cyc,w}) $.
Since $ F^\prime $ is unramified outside $ p, $ the reduction type at $ v $ does not change for $ E_1 $ and $ E_2 $ over $ F^\prime $ \cite[Chapter 7, Proposition 5.4]{Si}.
{\bl Since $ p \geq 5, $ from \cite[Proposition 5.1]{HM} we have $ E_2[p^\infty](F_v(\mu_{p^{\infty}}))=0 $ hence $ E_2(F^\prime_{cyc,w})[p] = 0 $.
{\cl It follows from} } (\ref{e2.4}) that $$ \overline{\rho}_{E_1,p}\mid_{G^{\prime\prime}} \enspace   \sim
   \left( {\begin{array}{cc}
    1 & \psi \\
    0 & 1 \\
   \end{array} } \right)  $$ 
   which implies $ E_1(F^\prime_{cyc,w})[p] \neq 0, $ a contradiction to the fact that $ E_1[p] \cong E_2[p]. $
Hence it is not possible to have elliptic curves $ E_1 $ and $ E_2 $ with one having split multiplicative reduction and other having additive reduction at $ v $.

Next, let $ E_1 $ have non-split multiplicative reduction at $ v $ and $ F^{\prime\prime} $ be the quadratic, unramified extension of $ F $ such that $ E_1 $ has split multiplicative reduction at $ v $ over $ F^{\prime\prime} $.
The reduction type of $ E_2 $ at $ v $ does not change since $ F^{\prime\prime} $ is unramified.
{\cl However, the arguments in the preceding paragraph shows that it is not possible.
Hence there does not exist elliptic curves with the above reduction types at $ v. $}

\end{proof}

\begin{lemma}[\rd \bf Split]\label{l2.7}
{\rd Let $ v $ be a finite prime of $ F $ with $ v \nmid p $ and $ w $ be a prime of $ F_{cyc} $ with $ w \mid v $.}
Let $ E $ be an elliptic curve with split multiplicative reduction at $ v $. 
From the representation (\ref{e2.3}), we have 
\begin{equation}\label{e2.5}
\rho_{E,p}\mid_{G}  \enspace \sim
   \left( {\begin{array}{cc}
    \epsilon_p & * \\
    0 & 1 \\
   \end{array} } \right).
\end{equation}
Then $ \sigma_{E}^{w} =  0 $ if $ \epsilon_p \neq 1 $ otherwise  $ \sigma_{E}^{w}=1. $
\end{lemma}

\begin{proof}

Let $ \epsilon_p \neq 1 $ i.e. $ \mu_p \nsubseteq F_v $.
{\rd Recall that} $ T_p(E) $ denote the Tate module of $ E $ and $ V_p(E)= T_p(E) \otimes \Q_p $ be the  representation corresponding to (\ref{e2.5}). 
Then $ V_p(E) / T_p(E) \cong E[p^\infty]. $
Since $ \rho_{E,p} $ is {\cl ramified}, we have $ * \neq 0 $ which implies $ \hzero(G, V_p(E))=0. $ 
Consider the following exact sequence
$$ 0 \longrightarrow T_p(E) \longrightarrow V_p(E) \longrightarrow E[p^\infty] \longrightarrow 0. $$
Let $ J $ be the extension $ F_v(E[p^\infty]) $ of $ F_v $ obtained by adjoining the coordinates of points in $ E[p^\infty]. $
 For $ G^\prime=\text{Gal}(J/F_{cyc,w}) $, we have 
$$ \hzero(G^\prime,V_p(E)) \longrightarrow \hzero(G^\prime,E[p^\infty]) \longrightarrow \hone(G^\prime,T_p(E)).  $$
Since $ \hzero(G^\prime, V_p(E)) = \hzero(G, V_p(E))=0, \, \hone(G^\prime,T_p(E)) $ is a finitely generated $ \Z_p $-module and $ \hzero(G^\prime,E[p^\infty]) $ is a torsion $ \Z_p $-module, we have $ \hzero(G,E[p^\infty])=\hzero(G^\prime,E[p^\infty]) $ to be finite.
Now by arguments similar to the proof of \cite[Lemma 3.5]{Sh}, we get $ \hone(G,E[p^\infty])=0 $ i.e. $ \sigma_{E}^{w}=0. $
For $ \epsilon_p = 1 $ i.e. $ \mu_p \subseteq {\rd F_v} $, the proof of $ \sigma_{E}^{w}=1 $ is given in \cite[Lemma 3.3]{Sh}.

\end{proof}

\begin{lemma}[\rd \bf Non-split]\label{l2.8}
{\rd Let $ v $ be a finite prime of $ F $ with $ v \nmid p $ and $ w $ be a prime of $ F_{cyc} $ with $ w \mid v $.}
Let $ E $ be an elliptic curve with non-split multiplicative reduction at $ v $. 
Then $ \sigma_{E}^{w} =  0 $ if $ \epsilon_p \neq \delta $ otherwise  $ \sigma_{E}^{w}=1. $
\end{lemma}

\begin{proof}

From {\rd the representation} (\ref{e2.4}), we have
\begin{equation}\label{e2.6}
   \overline{\rho}_{E,p}\mid_{G}   \enspace \sim
    \left( {\begin{array}{cc}
      \epsilon_p & \psi \\
     0 & 1 \\
    \end{array} } \right) \otimes \delta.
\end{equation}
Let $ \epsilon_p \neq \delta $.
{\rd It now follows that} $ \hzero(G, E[p])=0 $ i.e. $ \hzero(G, E[p^{\infty}])=0 $.
Now by an argument similar to the proof of \cite[Lemma 3.5]{Sh}, we have $ \hone(G,E[p^\infty])=0 $ i.e. $ \sigma_{E}^{w}=0. $
Next assume $ \epsilon_p = \delta $. 
Since $ \delta $ is a quadratic character, from {\rd the representation} (\ref{e2.3}) we have 
\begin{equation}\label{e2.7}
\rho_{E,p}\mid_{G}  \enspace \sim
    \left( {\begin{array}{cc}
     1 & * \\
     0 & \delta \\
    \end{array} } \right).
\end{equation}
As $ \delta $ is non-trivial, {\rd the above representation gives} $ 1=\text{rank}_{\Q_p} \, \hzero(G, V_p(E))=\text{corank}_{\Z_p} \, \hzero(G, E[p^{\infty}]) $.
Finally, {\rd from \cite[Proposition 2]{Gr} we have} $  \text{corank}_{\Z_p} \, \hzero(G, E[p^{\infty}]) = \text{corank}_{\Z_p} \, \hone(G, E[p^{\infty}]) $ {\rd which implies} $ \sigma_{E}^{w}=1. $

\end{proof}

\begin{lemma}[\rd \bf Split \textendash \ Non-split]\label{l2.9}
{\rd Let $ v $ be a finite prime of $ F $ with $ v \nmid p $ and $ w $ be a prime of $ F_{cyc} $ with $ w \mid v $.}
Let $ E_1 $ and $ E_2 $  be two elliptic curves such that $ E_1 $ has split multiplicative reduction at $ v, E_2 $ has non-split multiplicative reduction at $ v $ and $ E_1[p] \cong E_2[p], $ then $ \sigma_{E_1}^{w} \equiv \sigma_{E_2}^{w} + 1 \,(\text{mod} \enspace 2)$. 
\end{lemma}

\begin{proof}
Let $ \epsilon_p = \delta $.
Since $ \delta $ is non-trivial, we have $ \epsilon_p  \neq 1$.
Then the {\cl claim} follows from Lemmas \ref{l2.7} and \ref{l2.8}. 
Next assume $ \epsilon_p \neq \delta $.
From {\rd the representation} (\ref{e2.4}), we have 
$$ \overline{\rho}_{E_1,p}\mid_{G}  \enspace \sim
   \left( {\begin{array}{cc}
    \epsilon_p & \psi \\
    0 & 1 \\
   \end{array} } \right) $$
and $$ \overline{\rho}_{E_2,p}\mid_{G}   \enspace \sim
    \left( {\begin{array}{cc}
     \epsilon_p & \psi \\
     0 & 1 \\
    \end{array} } \right) \otimes \delta $$
Since $ E_1[p] \cong E_2[p] $, we have $ (\text{Trace} \, \overline{\rho}_{E_1,p})({\text{Frob}_v}) = (\text{Trace} \, \overline{\rho}_{E_2,p})({\text{Frob}_v}) $.
{\rd As} $ \delta $ is non-trivial, {\rd it follows that} $ \epsilon_p({\text{Frob}_v})=-1 $.
Therefore $ \epsilon_p $ is also a quadratic character and $ {\rd [F_v(\mu_p):F_v]}=2 $. 
Consider the field $ F^\prime:={\rd F_v(\mu_p)} $ and $ G^{\prime\prime}:=\text{Gal}({\rd \overline{F}_{v}}/F^\prime_{cyc,w}). $
Then from {\rd the representation} (\ref{e2.4}), we have 
\begin{equation}\label{e2.8}
\overline{\rho}_{E_1,p}\mid_{G^{\prime\prime}} \enspace   \sim
   \left( {\begin{array}{cc}
    1 & \psi \\
    0 & 1 \\
   \end{array} } \right) 
\end{equation}  and
\begin{equation}\label{e2.9}  
\overline{\rho}_{E_2,p}\mid_{G^{\prime\prime}} \enspace   \sim
   \left( {\begin{array}{cc}
    1 & \psi \\
    0 & 1 \\
   \end{array} } \right) \otimes \delta 
\end{equation}
Since $ \epsilon_p \neq \delta $ and $ [F^\prime:{\rd F_v}]=2 $, we have $ \delta \mid_{F^\prime} \neq 1 $.
Therefore {\rd the representations} (\ref{e2.8}) and (\ref{e2.9}) imply $ E_1(F^\prime_{cyc,w})[p] \neq 0 $ but $ E_2(F^\prime_{cyc,w})[p] = 0 $ which contradicts the fact that $ E_1[p] \cong E_2[p]. $
Hence the case $ \epsilon_p \neq \delta $ is not possible.

\end{proof}

\begin{remark}
Let $ p\geq5 $ and {\rd $ l $ be a finite prime of $ \Q $ with $ l \nmid p $.}
Since $ \epsilon_p \neq \delta $, it is not possible to have elliptic curves $ E_1 $ and $ E_2 $ such that $ E_1 $ has split multiplicative reduction at $ l, E_2 $ has non-split multiplicative reduction at $ l $ and $ E_1[p] \cong E_2[p]. $
\end{remark}

\begin{lemma}[\rd \bf Good \textendash \ Split]\label{l2.11}
{\rd Let $ v $ be a finite prime of $ F $ with $ v \nmid p $ and $ w $ be a prime of $ F_{cyc} $ with $ w \mid v $.}
Let $ E_1 $ and $ E_2 $ be two elliptic curves defined over $ F $ such that $ E_1 $ has good reduction at $ v $, $ E_2 $ has split multiplicative reduction at $ v $ and $ E_1[p] \cong E_2[p], $ then $ \sigma_{E_1}^{w} \equiv \sigma_{E_2}^{w} + 1 \,(\text{mod} \enspace 2)$.
\end{lemma}

\begin{proof}

Let $ \epsilon_p = 1 $ i.e. $ \mu_p \subseteq F_v $.
Then the {\cl claim} follows from \cite[Lemma 3.3 and Lemma 3.6]{Sh}.
Next assume $ \epsilon_p \neq 1 $.
From Lemma \ref{l2.7}, we have $ \sigma_{E_2}^{w}=0. $
Consider the field $ F^\prime:={\rd F_v(\mu_p)} $ and $ G^{\prime\prime}:=\text{Gal}({\rd \overline{F}_{v}}/F^\prime_{cyc,w}) $.
Then from {\rd the representation} (\ref{e2.4}), we have 
$$ \overline{\rho}_{E_2,p}\mid_{G^{\prime\prime}} \enspace   \sim
   \left( {\begin{array}{cc}
    1 & \psi \\
    0 & 1 \\
   \end{array} } \right) $$ 
which implies $ E_2(F^\prime_{cyc,w})[p] \neq 0 $ {\rd and hence $ E_1(F^\prime_{cyc,w})[p] \neq 0. $ }
{\rd It follows} from \cite[Proposition 5.1]{HM} {\rd that} $ E_1(F^\prime_{cyc,w})[p] \cong \Q_p/\Z_p \oplus \Q_p/\Z_p. $
Therefore $$\overline{\rho}_{E_1,p}\mid_{G^{\prime\prime}} \enspace   \sim
   \left( {\begin{array}{cc}
    1 & 0 \\
    0 & 1 \\
   \end{array} } \right) $$
Since order of $ \text{Gal}(F^\prime_{cyc,w}/F_{cyc,w}) $ is prime to $ p, $ $ E_1[p] $ is a semisimple $ G $-module i.e. 
\begin{equation}\label{e2.10}
\overline{\rho}_{E_1,p}\mid_{G}  \enspace \sim
   \left( {\begin{array}{cc}
    \epsilon_p & 0 \\
    0 & 1 \\
   \end{array} } \right)
\end{equation}
Since $ \epsilon_p \neq 1 $, {\rd it now follows that} $ \hzero(G, E_1[p^\infty])[p]=\hzero(G, E_1[p]) \cong \Z/p\Z. $
{\rd Further as} $  \hzero(G, E_1[p^{\infty}]) $ is {\rd $ p $-divisible} \cite[Lemma 4.1.2]{EPW}, we have $ \hzero(G, E_1[p^{\infty}]) \cong \Q_p/\Z_p. $
Finally, {\rd from \cite[Proposition 2]{Gr} we have} $  \text{corank}_{\Z_p} \, \hzero(G, E_1[p^{\infty}]) = \text{corank}_{\Z_p} \, \hone(G, E_1[p^{\infty}]) $ {\rd which implies} $ \sigma_{E_1}^{w}=1. $

\end{proof}

\begin{lemma}[\rd \bf Good \textendash \  Non-split]\label{l2.12}
{\rd Let $ v $ be a finite prime of $ F $ with $ v \nmid p $ and $ w $ be a prime of $ F_{cyc} $ with $ w \mid v $.}
Let $ E_1 $ and $ E_2 $  be two elliptic curves such that $ E_1 $ has good reduction at $ v, E_2 $ has non-split multiplicative reduction at $ v $ and $ E_1[p] \cong E_2[p], $ then $ \sigma_{E_1}^{w} =  \sigma_{E_2}^{w} $. 
\end{lemma}

\begin{proof}

Let $ \epsilon_p \neq \delta $ { \rd in the representation (\ref{e2.4}) }.
Then $ \hzero(G, E_2[p^{\infty}])=0 $.
{\rd It follows from the exact sequence (\ref{e2.2}) that} $ \hone(G, E_2[p]) \cong \hone(G, E_2[p^{\infty}])[p] $.
Since $ E_1 $ has good reduction at $ v $, by similar arguments as in Lemma \ref{l2.1} we have $ \hone(G, E_1[p]) \cong \hone(G, E_1[p^{\infty}])[p]. $ 
Therefore $ E_1[p] \cong E_2[p] $ implies $ \sigma_{E_1}^{w} =  \sigma_{E_2}^{w}. $
Next assume $ \epsilon_p = \delta $.
Then $$ \overline{\rho}_{E_2,p}\mid_{G}  \enspace \sim
    \left( {\begin{array}{cc}
     1 & \psi \\
     0 & \delta \\
    \end{array} } \right) $$
which implies $ \hzero(G, E_2[p]) \neq 0. $
Since $ E_1[p] \cong E_2[p], $ we have $ \hzero(G, E_1[p]) \neq 0. $
Also $ \delta $ is non-trivial implies $ \hzero(G, E_1[p^\infty])[p]=\hzero(G, E_1[p]) \cong \Z/p\Z. $
Since $  \hzero(G, E_1[p^{\infty}]) $ is divisible \cite[Lemma 4.1.2]{EPW}, we have  $  \hzero(G, E_1[p^{\infty}]) \cong \Q_p/\Z_p. $
Finally, {\rd by \cite[Proposition 2]{Gr} we have} $  \text{corank}_{\Z_p} \, \hzero(G, E_1[p^{\infty}]) = \text{corank}_{\Z_p} \, \hone(G, E_1[p^{\infty}]) $ {\rd from which it follows that} $ \sigma_{E_1}^{w}=1. $
Also from Lemma \ref{l2.8}, we have $ \sigma_{E_2}^{w}=1. $

\end{proof}

\section{$ p $-Selmer rank for ordinary elliptic curves} \label{s3}

Let  $ p $ be an odd prime, $ F $ be a number field and $ F_{cyc}/F $ denote the cyclotomic $ \Z_p $-extension of $ F $ with $ n $-th layer $ F_n $.
{\cl We put $ \Gamma_n :=\text{Gal}(F_n/F) $ and $ \Gamma := \text{Gal}(F_{cyc}/F). $ 
Then $ \Gamma_n \cong \Z/p^n\Z $ and $ \Gamma \cong \Z_p. $
The Iwasawa algebra $ \Lambda  :=\mathbb{Z}_{p}[[\Gamma]] \cong \mathbb{Z}_{p}[[T]] $, where the isomorphism is given by identifying $ 1 + T $ with a topological generator of $ \Gamma. $}

Let $ E_1 $ and $ E_2 $ be two elliptic curves defined over $ \Q $ with good, ordinary reduction at $ p. $ 
Let $ N_i $ denote the conductor of $ E_i $ over $ F $ and $ \overline{N_i} $ denote the prime to $ p $-part of the conductor of the Galois module $ E_i[p] $ over $ F $ for $ i = 1, 2. $ 
Let $ \Sigma $ be the finite set of primes of $ F $ containing the primes of bad reduction of $ E_1 $ and $ E_2 $, the infinite primes and the primes above $ p. $ 
Put,
\begin{equation}\label{e3.1}
 \Sigma_0 := \{v \in \Sigma \mid v \mid N_1/\overline{N_1} \enspace \text{or} \enspace v \mid N_2/\overline{N_2}\}.
\end{equation}
Let $ S_{E_i } $ denote the set of primes $ v \in \Sigma_0 $ such that $ E_i $ has split multiplicative reduction at $ v $ and $ s_p(E_i/F) $ denote the 
{\cl $ \Z_p$-}rank of the Pontryagin dual of $ \text{Sel}_p(E_i/F) $, for $ i = 1, 2. $
Then in \cite{Sh}, the author has studied the variation of the parity of the $ p $-Selmer ranks of $ E_1 $ and $ E_2 $ over $ F. $
He also proved the following.
\begin{theorem}\cite[Theorem 1.1]{Sh}
Let $ E_1 $ and $ E_2 $ be two elliptic curves defined over $ \Q $ with good and ordinary reduction at an odd prime $ p. $ 
We further assume that the following hold

(a) $ E_1[p] $ is an irreducible $ \text{Gal}(\overline{F} /F) $-module.

(b) $  \mu_p \subseteq F $ where $ \mu_p $ denotes the group of $ p $-th roots of unity.

(c) $ \sel_{p}(E_1/F_{cyc})[p] $ is finite where $ F_{cyc} $ denote the cyclotomic $ \Z_p $-extension of $ F. $ 

If $ E_1 $ and $ E_2 $ are congruent at the prime $ p $ then
\begin{equation}\label{e3.2}
s_p(E_1/F) + \mid S_{E_1} \mid \equiv s_p(E_2/F) + \mid S_{E_2} \mid (\text{ mod} \enspace 2).
\end{equation}

\end{theorem}

{\tl If $ E_1 $ and $ E_2 $ are congruent at $ p $, then the assumption (a) and (c) hold for $ E_1 $ if and only if the same hold for $ E_2 $.
{\cl Let the Pontryagin dual $ X(E/F^{cyc}) $ of $ Sel(E[p^{\infty}]/F^{cyc}) $ is a finitely generated torsion $ \Lambda $-module. 
Then by the structure theorem of finitely generated $ \Lambda $-modules, one has a pseudo isomorphism
$$X(E/F^{cyc}) \sim (\oplus_{i=1}^{s}\Lambda/(f_{i}(T)^{a_{i}}))\oplus(\oplus_{j=1}^{t}\Lambda/(p^{\mu^{(j)}_{E}}))$$
where $s,t,a_{i},\mu^{(j)}_{E} \in \mathbb{N}$, $f_{i}$ is distinguished and irreducible for all $i$.
Since, the $ a_{i} $'s and the $ \mu^{(j)}_{E} $'s are positive integers, one can define the algebraic Iwasawa invariants $ \lambda_{E} $ and $ \mu_{E} $ by $$\lambda_{E} = \sum_{i=1}^{s}a_i \, \text{deg}(f_{i}(T)), \hspace{.3cm} \mu_{E} = \sum_{j=1}^{t} \mu^{(j)}_{E}.$$ }
{\cl In the case $ F=\Q $ and $ p $ is a prime of} ordinary reduction {\cl for $ E $}, it is a conjecture of Greenberg that there exists a $ \Q $-isogenous elliptic curve $ E^{\prime} $ such that $ \mu_{E^{\prime}} = 0 $ \cite[Conjecture 1.11]{Gr1}. 
In particular, if $ E[p] $ is irreducible as a $ \Z/p\Z $-representation of $\text{Gal}(\overline{\Q}/\Q)$, {\cl it is believed that} $ \mu_{E} = 0 $.
However, over general number fields none of the elliptic curves isogenous to $ E $ may have $ \mu $-invariant equal to zero \cite{Dr} and also there are no examples known for which assumption (a) holds but assumption (c) does not hold.}

Next assume $ \mu_p \nsubseteq F. $ 
Then for $ p=3 $, we have the following counterexamples to the above theorem. 
\vspace{.2cm}

\noindent {\bf Example 1.} Consider the following elliptic curves:
\vspace{.2cm}

$ E_{1}: y^{2}+y=x^{3}-x^{2}-10x-20, \enspace (11a1) $
\vspace{.2cm}
 
$ E_{2}: y^{2}+xy=x^{3}+x^2-2x-7, \enspace (121c1) $
 \vspace{.2cm}
 
The labels in the parentheses above denote the Cremona number of the elliptic curves.
$E_1$ has split multiplicative reduction at 11 while $E_2$ has additive reduction at 11.
Both the curves have $ \lambda $-invariant 0 and $ E_1[3] \cong E_2[3] $.
Hence (\ref{e3.2}) does not hold. 
\vspace{.2cm}

\noindent {\bf Example 2.} Consider the following elliptic curves which are obtained by twisting the curves in the above example by a quadratic character associated to the quadratic extension $ \Q(\sqrt{-5})/\Q $.
\vspace{.2cm}

$ E_{3}: y^{2}=x^{3}-x^{2}-4133x+186637, \enspace (4400m2) $
\vspace{.2cm}
 
$ E_{4}: y^{2}=x^{3}-x^2-1008x+48512, \enspace (48400ch1) $
 \vspace{.2cm}
 
$E_3$ has non-split multiplicative reduction at 11 and additive reduction at 2,5 while $E_4$ has additive reduction at 2,5,11.
$E_3$ and $E_4$ have $ \lambda $-invariant 0 and 1 respectively and $ E_3[3] \cong E_4[3] $.
Hence (\ref{e3.2}) does not hold.
Nevertheless, we have the following theorem.

\begin{theorem} \label{t3.2}
Let $ E_1 $ and $ E_2 $ be two elliptic curves defined over $ \Q $ with good and ordinary reduction at an odd prime $ p. $ 
We further assume the following.

(a) $ E_1[p] $ is an irreducible $ \text{Gal}(\overline{F} /F) $-module.

(b) $ \sel_{p}(E_1/F_{cyc})[p] $ is finite.

If $ E_1[p] \cong E_2[p] $ as $ \text{Gal} \, (\overline{\Q}/\Q) $ modules then
\begin{equation}
s_p(E_1/F) + \mid S_{E_1} \mid \equiv s_p(E_2/F) + \mid S_{E_2} \mid + \mid T \mid (\text{ mod} \enspace 2)
\end{equation}
where $T$ denotes the set of primes $v$ of $F$ in $\Sigma_0$ such that $\mu_p\not\subset F_v$, one of the elliptic curves has multiplicative reduction and other has additive reduction at $v$.

\end{theorem}

\begin{proof}

Under the assumptions (a) and (b), from \cite[Equations 8 and 10]{Sh} we have $ \lambda^{\Sigma_0}_{E_1}=\lambda^{\Sigma_0}_{E_2} $ .
For a prime $ v \nmid p $ of $ F $, let $ \tau_{E_1}^{v} $ denote the $ \Z_p $-rank of the Pontryagin dual of $ \h_v(F_{\text{cyc}},E_1) $ where $$ \h_v(F_{\text{cyc}},E_1):=\prod_{w \mid v} \hone(F_{\text{cyc},w}, E_1[p^{\infty}]). $$
From \cite[Lemma 3.1]{Sh}, we have $ \lambda^{\Sigma_0}_{E_i}=\lambda_{E_i}+\sum_{v \in \Sigma_0}\tau_{E_i}^{v}, i=1,2.  $
Now for each $ v \in \Sigma_0 $, choose a prime $ w $ of $ F_{\text{cyc}} $ dividing $ v $.
Let $ P $ be the set consisting of all such $ w. $
Then \cite[Lemma 3.2]{Sh} implies $$ \lambda^{\Sigma_0}_{E_i} \equiv \lambda_{E_i}+\sum_{w \in P}\sigma_{E_i}^{w} \enspace (\text{ mod} \enspace 2), \ i=1,2.  $$
Therefore $$ \lambda_{E_1}+\sum_{w \in P}\sigma_{E_1}^{w} \equiv \lambda_{E_2}+\sum_{w \in P}\sigma_{E_2}^{w} \enspace (\text{ mod} \enspace 2).  $$

Let $ v \nmid p $ be a finite prime of $ F $ such that $ E_1 $ and $ E_2 $ have same reduction type at $ v $. 
Then Lemmas \ref{l2.1}, \ref{l2.2}, \ref{l2.7} and \ref{l2.8} imply $ \sigma_{E_1}^{w} = \sigma_{E_2}^{w}. $
Lemma \ref{l2.3} implies that it is not possible to have a prime $ v $ such that $ E_1 $ has good reduction and $ E_2 $ has additive reduction at $ v $.
Next suppose $ E_1 $ has split multiplicative reduction and $ E_2 $ has either good or non-split multiplicative reduction at $ v $, then  Lemma \ref{l2.9} and Lemma \ref{l2.11} imply $ \sigma_{E_1}^{w} \equiv \sigma_{E_2}^{w} + 1 \,(\text{mod} \enspace 2)$.
Also if $ E_1 $ has good reduction and $ E_2 $ has non-split multiplicative reduction at $ v $, then $ \sigma_{E_1}^{w} =  \sigma_{E_2}^{w} $ follows from Lemma \ref{l2.12}. 

Let $ p \geq 5 $.
By Lemma \ref{l2.10}, it is not possible to have a prime $ v $ such that $ E_1 $ has multiplicative reduction and $ E_2 $ has additive reduction at $ v $. 
Therefore assume $ p=3, E_1 $ has split multiplicative reduction and $ E_2 $ has additive reduction at $ v. $
From Lemma \ref{l2.7}, we have $ \sigma_{E_1}^{w}=1 $ if $ \mu_p \subset F_v $ otherwise  $ \sigma_{E_1}^{w}=0. $
Since $ \sigma_{E_2}^{w}=0 $, we have $ \sigma_{E_1}^{w} \equiv \sigma_{E_2}^{w} + 1 \,(\text{mod} \enspace 2)$ if $ \mu_p \subset F_v $ otherwise  $ \sigma_{E_1}^{w}=\sigma_{E_2}^{w}=0. $
Next, assume $ E_1 $ has non-split multiplicative reduction at $ v. $
We have $ \epsilon_p \neq \delta $ if $ \mu_p \subset F_v $ otherwise $ \epsilon_p = \delta $ since $ p=3 $ and there exists a unique unramified quadratic character.
Hence Lemma \ref{l2.8} implies $ \sigma_{E_1}^{w}=0 $ if $ \mu_p \subset F_v $, else  $ \sigma_{E_1}^{w}=1. $
Since $ \sigma_{E_2}^{w}=0 $, we have $ \sigma_{E_1}^{w}=\sigma_{E_2}^{w}=0 $ if $ \mu_p \subset F_v $ otherwise $ \sigma_{E_1}^{w} \equiv \sigma_{E_2}^{w} + 1 \,(\text{mod} \enspace 2)$.

From the above comparisons of $ \sigma_{E_i}^{w}, i=1,2 $, we have
$$ \lambda_{E_1} + \mid S_{E_1} \mid \equiv \lambda_{E_2} + \mid S_{E_2} \mid + \mid T \mid (\text{ mod} \enspace 2). $$

{\cl By \cite[Proposition 3.10]{Gr1}, we have} $ s_p(E_i/F) \equiv \lambda_{E_i} (\text{ mod} \enspace 2), \ i=1,2. $
Therefore $$ s_p(E_1/F) + \mid S_{E_1} \mid \equiv s_p(E_2/F) + \mid S_{E_2} \mid + \mid T \mid (\text{ mod} \enspace 2). $$

\end{proof}

\section{$ p $-Selmer rank for supersingular elliptic curves} \label{s4}
Suppose $ E $ be an elliptic curve defined over $ \Q $ with good supersingular reduction at $ p $ and $ a_p(E):=1+p-\#\widetilde{E}(\F_p)=0. $ 
{\gl The assumption that $ a_p(E) = 0 $ is necessary to apply some crucial results of Kim \cite{Ki1}.
However for $ p \geq 5, \, E $ is supersingular at $ p $ if and only if $ a_p(E)=0 $ \cite[Exercise 5.10(b)]{Si}.}
{\rd Throughout this section, let $ F $ be a number field such that $ p $ splits completely over $ F. $}

\begin{definition}\cite[Section 2, Plus/Minus Selmer group]{Ko}
For $ n \geq 0 $, we define $ \sel^{\pm}_p(E/F_n) $ to be the kernel of 
\begin{equation} \label{e1.1}
{\bf f_n:} \enspace \hone(F_n, {\cl E[p^\infty]}) \longrightarrow \prod_{v}\dfrac{\hone(F_{n,v},{\cl E[p^\infty]})}{E^{\pm}(F_{n,v}) \otimes \Q_p/\Z_p}
\end{equation}
 and $ \sel^{\pm}_p(E/F_{cyc}):=\ilim_n \sel^{\pm}_p(E/F_n). $
Here $ v $ ranges over all places of $ F_n $, $ F_{n,v} $ is the completion of $ F_n $ at $ v $ and the map is induced by the restrictions of the Galois cohomology groups.
For notational convenience, let $F_{-1,v}=F_{0,v}=F_v. $
Then $$ E^{+}(F_{n,v}) := \{ P \in E(F_{n,v}) \mid \text{Tr}_{n/m+1}P \in E(F_{m,v}) \enspace \text{for even} \enspace m \enspace (0 \leq m < n) \} $$
$$ E^{-}(F_{n,v}) := \{ P \in E(F_{n,v}) \mid \text{Tr}_{n/m+1}P \in E(F_{m,v}) \enspace \text{for odd} \enspace m \enspace (-1 \leq m < n) \} $$
where $ \text{Tr}_{n/m}: E(F_{n,v}) \longrightarrow E(F_{m,v}) $ denotes the trace map for $ n \geq m $.

\end{definition}

\begin{remark}\label{rem}
For $ n=0 $, we have $E^{-}(F_{v})=E(F_{v}) $ hence $ \sel^{-}_p(E/F)=\sel_p(E/F). $ 
\end{remark}

Let $ E_1 $ and $ E_2 $ are two elliptic curves defined over $ \Q $ with good, supersingular reduction at an odd prime $ p $ and $ a_p (E_i)=0, i=1,2. $ 
{\rd Henceforth we make the following assumptions.}
\begin{assumption} \label{a4.1}
For $ E=E_1 \, \text{or} \, E_2$
\vspace{.2cm}

$ (\romannumeral 1) \, E[p] $ is an irreducible $ \text{Gal}(\overline{F} /F) $-module. 
\vspace{.1cm}

$ (\romannumeral 2) \, \sel^{-}_p(E/F_{cyc})[p] $ is a finite group.
\vspace{.1cm}
\end{assumption}
{\gl For $ F=\Q $, assumption $ (\romannumeral 1) $ is true \cite{Se} and assumption $ (\romannumeral 2) $ is conjectured to be true \cite[Conjecture 7.1]{Pe}.}
If assumption $ (\romannumeral 2) $ holds, then the Pontryagin dual $ X^{-}(E/F_\text{cyc}) $ of $ \sel^{-}_p(E/F_\text{cyc}) $ is a finitely generated $ \Z_p $-module and therefore also a $ \Lambda $-torsion module. 
Since $ p $ splits completely over $ F/\Q $, the Selmer group $ \sel^{-}_p(E/F_\text{cyc}) $ is $ p $-divisible \cite[Theorem 3.14]{Ki2}. 

Next we define the {\it imprimitive} plus/minus Selmer group for $ \Sigma_0 $ in (\ref{e3.1}). 
Recall that $\Sigma$ is a finite set of primes of $F$ containing primes of bad reduction of $E$, the primes dividing $p$ and infinite primes. 
For $ l \in \Sigma $ with $ l \nmid p $, we define a conventional local condition 
$$ \mathcal{H}_{l}^{\pm}(F_\text{cyc},E[p^{\infty}]) := \prod_{\eta \mid l}\dfrac{\hone(F_{\text{cyc},\eta},E[p^{\infty}])}{E(F_{\text{cyc},\eta}) \otimes \Q_p/\Z_p}. $$
Note that $ {E(F_{\text{cyc},\eta}) \otimes \Q_p/\Z_p}=0 $ if $ \eta $ is not above $ p. $
We will also let $ \mathcal{H}_{l}(F_\text{cyc},E[p^{\infty}]) $ denote $ \mathcal{H}_{l}^{\pm}(F_\text{cyc},E[p^{\infty}]) $ to emphasize $ \h_{l}^+ = \h_{l}^- $ when $ l \nmid p. $
For $ v \mid p $, put  
$$ \mathcal{H}_{v}^{\pm}(F_\text{cyc},E[p^{\infty}]) := \dfrac{\hone(F_{\text{cyc},v},E[p^{\infty}])}{E^{\pm}(F_{\text{cyc},v}) \otimes \Q_p/\Z_p} $$  

\begin{definition}\cite[Page 185]{Ki1}
$$ \sel^{\Sigma_0,\pm}_p(E/F_{\text{cyc}}) = \text{Ker} \left(\hone(F_\Sigma/F_\text{cyc}, E[p^{\infty}]) \longrightarrow \prod_{l \in \Sigma-\Sigma_0} \mathcal{H}_{v}^{\pm}(F_\text{cyc},E[p^{\infty}]) \right). $$
\end{definition}
Here $ F_\Sigma $ is the maximal extension of $ F $ unramified outside $ \Sigma$. 
We mention that when $\Sigma_0$ is empty, the imprimitive plus/minus Selmer group defined above is same as the usual plus/minus Selmer group defined in (\ref{e1.1}). 
These {\rd plus/minus Selmer groups fit into the following exact sequence }
\begin{equation}\label{e4.1}
0 \longrightarrow {\cl \sel^-_p}(E/F_{\text{cyc}}) \longrightarrow {\cl \sel^{\Sigma_0,-}_p}(E/F_{\text{cyc}}) \longrightarrow \prod_{v \in \Sigma_0} \mathcal{H}_{v}(F_{\text{cyc}},E[p^\infty]) \longrightarrow 0.
\end{equation}
{\rd The exactness of the above sequence is proven for $ F=\Q $ and a larger set $ \Sigma_0 $} in \cite[Corollary 2.5]{Ki1}.
{\cl Under the Assumption $ \ref{a4.1} (\romannumeral 2) $ that} $ {\cl \sel^-_p}(E/F_\text{cyc}) $ is $ \Lambda $-cotorsion, similar assertion holds for our chosen $ F $ and $ \Sigma_0 $ \cite[Section 2, Page 3259]{Sh}.
It is well known that $ \mathcal{H}_{v}(F_{\text{cyc}},E[p^\infty]) $  is a divisible group for every finite prime $ v $ of $ F $ \cite[Lemma 4.5 and successive paragraph]{Gr1}.
Hence $ {\cl \sel^{\Sigma_0,-}_p}(E/F_{\text{cyc}}) $ is a divisible group.
Therefore $$ {\cl \lambda^{\Sigma_0, -}_{E}} = \text{rank}_{\Z/p\Z} \, {\cl X^{\Sigma_0, -}}(E/F_\text{cyc})/p = \text{rank}_{\Z/p\Z} \, {\cl \sel_p^{\Sigma_0, -}}(E/F_\text{cyc})[p]. $$
Next we define a "plus/minus {\it imprimitive} Selmer" group on $ E[p]. $ 
For $ l \nmid p $, we let $$ \mathcal{H}_{l}^{\pm}(F_\text{cyc},E[p])=\prod_{\eta \mid l}\hone(F_{\text{cyc},\eta},E[p])/\hone_{un}(F_{\text{cyc},\eta},E[p]) $$ {\cl where $$ \hone_{un}(F_{\text{cyc},\eta},E[p]) = \ilim_n \, \text{Ker} \left( \hone(F_{n,l}, E[p]) \longrightarrow \hone(I_{n,l}, E[p]) \right)  $$
and $ I_{n,l} $ denotes the inertia group of the decomposition group of $ F_{n,l}. $ }
We also let $ \mathcal{H}_l $ denote $ \mathcal{H}_l^{\pm} $ to emphasize $ \mathcal{H}_l^+ = \mathcal{H}_l^- $ {\cl when $ l \neq p $}.
For $ v|p $, we let $$ \mathcal{H}_{v}^{\pm}(F_\text{cyc},E[p])=\hone(F_{\text{cyc},v},E[p])/(E^{\pm}(F_{\text{cyc},v})/pE^{\pm}(F_{\text{cyc},v})). $$

\begin{definition} \cite[Definition 2.7]{Ki1}
$$ \sel^{\Sigma_0,\pm}(E[p]/F_{\text{cyc}}) = \text{Ker} \left(\hone(F_\Sigma/F_\text{cyc}, E[p]) \longrightarrow \prod_{l \in \Sigma-\Sigma_0} \mathcal{H}_{l}^{\pm}(F_\text{cyc},E[p]) \right). $$
\end{definition}
{\cl Using analogous arguments from \cite[Section 2, Page 3259]{Sh}, one can extend \cite[Proposition 2.10]{Ki1} for our chosen $ F $ and $ \Sigma_0 $ to prove that $ \sel^{\Sigma_0,-}(E[p]/F_{\text{cyc}}) \cong \sel_p^{\Sigma_0,-}(E/F_{\text{cyc}})[p] $. 
Hence $$ \lambda^{\Sigma_0, -}_{E}= \text{rank}_{\Z/p\Z} \enspace {\cl \sel}^{\Sigma_0, -}(E[p]/F_\text{cyc}). $$ }
Now suppose $ E_1[p] \cong E_2[p]. $
Let $ v \mid p $ be a finite prime of $ F $ and $ w \mid v $ be a prime of $ F_{cyc} $.
Since  $ p $ splits completely over $ F/\Q $, we have $ F_v=\Q_p. $
Therefore from \cite[Proposition 2.9]{Ki1}, it follows that {\cl $$\lambda^{\Sigma_0, -}_{E_1} = \lambda^{\Sigma_0, -}_{E_2}.$$}
Since Assumption $ \ref{a4.1} (\romannumeral 2) $ implies {\cl $ \sel^{-}_p(E/F_\text{cyc}) $} is $ \Lambda $-cotorsion, from \cite[Lemma 3.1]{Sh} we have
\begin{equation} \label{e4.3}
{\cl \lambda^{\Sigma_0, -}_{E_1} = \lambda^{-}_{E_i}} + \sum_{v \in \Sigma_0} \tau_{E_i}^{v} \enspace \text{for} \enspace i=1,2.
\end{equation}

To prove our theorem comparing Selmer rank of supersingular elliptic curves $E_1$ and $E_2$ at primes dividing $p$, we need the following. 

\begin{theorem}\label{par}

Let  $ E $ be an elliptic curve defined over a number field $F$ with good, supersingular reduction at primes dividing $ p $.
Assume that $ \sel_p^{-}(E/F_{cyc}) $ is $ \Lambda $-cotorsion. 
Then $$ \text{corank}_{\Z_p}(\sel_p^{-}(E/F)) \equiv \lambda^{-}_E \, (mod \; 2) $$
where $ \lambda^{-}_E$ denotes the Iwasawa $\lambda$-invariant of the Pontryagin dual of  $\sel_p^{-}(E/F_{cyc})$. 

\end{theorem}
The proof of the above theorem is analogous to the proof of a similar theorem proved by Guo for ordinary elliptic curves at primes dividing $p$ \cite{Gu1}. To prove his result in ordinary case, Guo first proves the following.

\begin{theorem}

For every elliptic curve $E$ defined over $F$, we have a nondegenerate, skew-symmetric pairing 
$$ \sel_p(E/F) \times \sel_p(E/F) \longrightarrow \Q_p/\Z_p $$
whose kernel on either side is precisely the maximal divisible subgroup.

\end{theorem}

Let $ v $ be a finite prime of $ F $ with $ v \mid p $ and $ D $ be a divisible subgroup of $ \hone(F_{v},{\cl E[p^\infty]}) $.
Define $$ J_n=\text{Ker} \, \{ \hone(F_{v},E[p^n]) \longrightarrow \hone(F_{v},{\cl E[p^\infty]})/D \} .$$
Then from the exact sequence $ 0\rightarrow E[p^r] \rightarrow E[p^{r+s}] \xrightarrow{p^r} E[p^s] \rightarrow 0 $, we have the induced exact sequence,
$$ J_r\rightarrow J_{r+s}\rightarrow J_s\longrightarrow 0. $$
for $ r,s $ large \cite[Proposition 1]{Gu1}. 
This is one of the key observations of Guo used in the proof of the above theorem. 
He considers $D$ to be equal to $E(F_v)\otimes \Q_p/\Z_p$. 
By replacing  $D=E(F_v)\otimes \Q/\Z_p$ with  $D= E^-(F_v)\otimes \Q_p/\Z_p$ for every prime $v \mid p$ and using arguments similar to \cite{Gu1}, we obtain the following. 

\begin{theorem}\label{pairing}

For every elliptic curve $E$ defined over $F$ with good and supersingular reduction at primes dividing $p$, we have a nondegenerate, skew-symmetric pairing 
$$ \sel_p^-(E/F) \times \sel_p^-(E/F) \longrightarrow \Q_p/\Z_p $$
whose kernel on either side is precisely the maximal divisible subgroup.

\end{theorem}

The proofs of nondegeneracy and skew-symmetric are included in \cite[Sections 3.3 \& 3.4]{Gu}.
{\rd In fact, for an arbitrary $ p $-ordinary Galois representation $ V $ with Tate module $ T $, Guo derived} the skew-symmetric property from the existence of a $ G_F $-module isomorphism $$ \eta: A \longrightarrow A^*=\text{Hom} \, (T, \Q_p/\Z_p(1)) \enspace {\rd \text{where} \enspace A:=V/T} $$ such that $ (\eta a)(a)=0 $ for $ a \in A. $
This isomorphism {\rd is eventually used to get} a pairing $$ e_s : A_s \times A_s \longrightarrow \Q_p/\Z_p(1) \enspace {\rd \text{where} \enspace A_s:=A[p^s]} $$ such that $ e_s(a,a^{\prime}) = -e_s(a^{\prime},a) $. 
For elliptic curves such a pairing, $ e_s $, {\rd exists already} from the Weil pairing.

Now Theorem \ref{par} can be proved easily using Theorem \ref{pairing}, the fact that $ E[p^\infty](F_{\text{cyc}}) $ is finite \cite[Theorem 1]{Ri} and methods of \cite{Gu1} (see also \cite[Lemmas 3 and 6, Proposition 6]{Gu1}.
Recall that for an elliptic curve $ E $ defined over $ F $, $ s_p(E/F) $ denotes the rank of the Pontryagin dual of $ \sel_p(E/F) $.
Then we have the following immediate corollary to Theorem \ref{par} and Remark \ref{rem}.
\begin{corollary}\label{c4.9}
Let $ E $ be an elliptic curve defined over a number field $F$ with good, supersingular reduction at primes dividing $ p $.
Assume that $ \sel_p^{-}(E/F_{cyc}) $ is $ \Lambda $-cotorsion. 
Then $$s_p(E/F) \equiv \lambda^{-}_E \, (mod \; 2) $$
\end{corollary}
\begin{theorem} \label{t4.11}
Let $ E_1 $ and $ E_2 $ be two elliptic curves defined over $ \Q $ with good and supersingular reduction at an odd prime $ p. $ 
We further assume the following.

(a) $ E_1[p] $ is an irreducible $ \text{Gal}(\overline{F} /F) $-module.

(b) $ \sel^{-}_p(E_1/F_{cyc})[p] $ is finite.

(c) $p$ splits completely over $F$.

If $ E_1[p] \cong E_2[p] $ as $ \text{Gal} \, (\overline{\Q}/\Q) $ modules then 
\begin{equation}
s_p(E_1/F) + \mid S_{E_1} \mid \equiv s_p(E_2/F) + \mid S_{E_2} \mid + \mid T \mid (\text{ mod} \enspace 2)
\end{equation}

\end{theorem}

\begin{proof}

Using (\ref{e4.3}) and similar arguments as in Theorem \ref{t3.2}, we have 
$$ \lambda^{-}_{E_1} + \mid S_{E_1} \mid \equiv \lambda^{-}_{E_2} + \mid S_{E_2} \mid + \mid T \mid (\text{ mod} \enspace 2). $$
{\rd Now} the theorem follows from Corollary \ref{c4.9}.

\end{proof}

{\dl For a slightly different formula (valid only over $ \Q $) comparing $ p $-Selmer rank of supersingular elliptic curves $ E_1 $ and $ E_2 $, see \cite{Ha}. } 

\section{Root numbers} \label{s5}

{\cl Let $ E $ be an elliptic curve defined over a number field $ F $ with good reduction at a prime $ v $ of $ F $ and $ V $ be the $ p $-adic Galois representation attached to $E$.
Let $ \varepsilon(V) $ denote the local epsilon-factor associated to $ V $ at $ v $.
Then from \cite{Do} and \cite[Proposition 2.2.1]{Ne1}, we have $ \varepsilon(V)=\pm 1 $.
Following \cite{Do}, the local root number, $ w(E/F_v) $, is given by
\begin{equation}\label{e5.3}
w(E/F_v):=\dfrac{\varepsilon(V)}{\mid \varepsilon(V) \mid}=\varepsilon(V)=\pm 1.
\end{equation}
\begin{remark}\cite[Section 3.4]{Do}
$ w(E/F_v)=-1 $ when $ v $ is Archimedean or when $ E/F_v $ has split multiplicative reduction, and $ w(E/F_v)=1 $ when $ E/F_v $ has good or non-split multiplicative reduction. 
\end{remark}
\noindent The global root number of an elliptic curve $ E $ defined over $ F $ is the product of the local root numbers over all places of $ F $,
$$ w(E/F) :=  \prod_v w(E/F_v). $$
Thus, if $ E/F $ is semistable, the global root number, $ w(E/F) $, is given by}
\begin{equation}\label{e5.1}
w(E/F):=\prod_{v}w(E/F_v)=(-1)^{\#\{v \, \mid \, \infty \enspace \text{in} \enspace F\} \, + \, \#\{v \enspace \text{split mult. for E\,/\,F} \}} 
\end{equation} 
{\gl {\cl Let $ v \nmid p $. Then we recall } the modified local constant, $ \varepsilon_0(V) $, as defined by Deligne (\cite[(5.1)]{De} and \cite[Page 6]{Ne3}) to be }
\begin{equation}\label{e5.4}
\varepsilon_0(V)=\varepsilon(V) \, \text{det}({\cl -\text{Frob}_{v}} \mid {\cl V^{I_{v}}})
\end{equation}
where the value of $ \varepsilon_0(V) $ mod $p$ depends only on the residual representation $ (T_pE)/p \cong E[p] $ \cite[Theorem 6.5]{De}. 
In the next few lemmas, we compare the global and local root number of two congruent elliptic curves $ E_1 $ and $ E_2 $ with different reduction types at $ v $.

\begin{theorem}\label{t5.1}

Let $ E_1 $ and $ E_2 $ be two semistable elliptic curves defined over $ \Q $ with good reduction at an odd prime $ p $ and $ E_1[p] \cong E_2[p]. $
Then $$ \dfrac{w(E_1/F)}{w(E_2/F)} = (-1)^{\mid S_{E_1} \mid - \mid S_{E_2} \mid}  $$

\end{theorem}

\begin{proof}

From (\ref{e5.1}), we have
\begin{equation}\label{e5.2}
\dfrac{w(E_1/F)}{w(E_2/F)} = (-1)^{\# \{v \enspace \text{split mult. for $ E_1 $\,/\,F}\} \,  - \, \# \{v \enspace \text{split mult. for $ E_2 $\,/\,F}\} }. 
\end{equation}
Let $ v \in $ \{Set of primes of split multiplicative reduction of $ E_1/F $\}  $ \setminus S_{E_{1}}. $
Therefore $ v \nmid N_1/\overline{N_1}. $
Since $ E_1 $ is semistable, we have $ (N_1/\overline{N_1}, \overline{N_1}) =1  $ which implies $ v \mid \overline{N_1}. $
{\rd It is clear that} $ \overline{N_1}=\overline{N_2} $ as $ E_1[p] \cong E_2[p] $.
Since $ E_2 $ is also semistable, we have $ v \nmid N_2/\overline{N_2} $ \, i.e. $ v \notin S_{E_{2}}. $

{\rd Note that} $ \hzero(F_{\text{cyc},w}, E_i[p^\infty]), i=1,2 $ is divisible for every prime $ w \mid v $ of $ F_{\text{cyc}} $ \cite[Lemma 2.1]{Sh}.
{\rd It follows from} (\ref{e2.2}) that $ \hone(G, E_i[p]) \cong \hone(G, E_i[p^{\infty}])[p]. $
{\rd Then the isomorphism} $ E_1[p] \cong E_2[p] $ {\rd gives} $ \sigma_{E_1}^{w} =  \sigma_{E_2}^{w}. $
Now from Section \ref{s2}, the parity of local Iwasawa invariants {\cl implies} that $ v $ is also a prime of split multiplicative reduction of $ E_2/F $.
Hence we have $ v \in $ \{Set of primes of split multiplicative reduction of $ E_2/F $\} $ \setminus S_{E_{2}} $ which implies $$ \# \{v \enspace \text{split mult. for} \enspace E_{1}/F \} \, - \mid S_{E_{1}} \mid = \# \{v \enspace \text{split mult. for} \enspace E_2/F\} \, - \mid S_{E_{2}} \mid. $$
Therefore {\rd the equality} in (\ref{e5.2}) reduces to $ \dfrac{w(E_1/F)}{w(E_2/F)} = (-1)^{\mid S_{E_1} \mid - \mid S_{E_2} \mid}.  $

\end{proof}

{\rd Combining} Lemmas \ref{l2.3} and \ref{l2.10}, we are only left with the following two cases.

\begin{lemma}\label{l5.2}

Let $ v \nmid p $ be a finite prime of $ F $.
Let $ E_1 $ and $ E_2 $ be two elliptic curves with additive reduction at $ v $ and $ E_1[p] \cong E_2[p]. $
Then $ w(E_1/F_v)=w(E_2/F_v) $.

\end{lemma}

\begin{proof}

{\bl  Let $ V_i $ be the Galois representation corresponding to $E_i, i=1,2$.
Since $ E_1[p] \cong E_2[p] $ and the value of $ \varepsilon_0(V_i) $ mod $p$ depends only on the residual representation $ E_i[p] $, we have $ \varepsilon_0(V_1) \equiv \varepsilon_0(V_2) \, \text{mod} \, p. $ }
{\gl Also $ E_i $ have additive reduction at $ v $ implies $ V_i^{{\cl I_v}}=0 $ \cite[Theorem 4.10.2]{Si1}.}
Therefore $ \varepsilon_0(V_i)=\varepsilon(V_i), i=1,2. $
Finally (\ref{e5.3}), (\ref{e5.4}) and $ \varepsilon_0(V_1) \equiv \varepsilon_0(V_2) \, \text{mod} \, p $ together imply $ w(E_1/F_v)=w(E_2/F_v) $.

\end{proof}

\begin{lemma}\label{l5.3}
Let $ p =3 $, $ v \nmid p $ be a finite prime of $ F $. 
Suppose $ E_1 $ and $ E_2 $ are two elliptic curves such that $ E_1 $ has multiplicative reduction at $ v, E_2 $ has additive reduction at $ v $ and $ E_1[p] \cong E_2[p] $.
Then $ w(E_2/F_v)=1 $ if $ \mu_p \subset F_v $ otherwise $ w(E_2/F_v)=-1. $
\end{lemma}

\begin{proof}

Let $ \mu_p \subset F_v $ which implies $ \epsilon_p $ is trivial.
Consider $ E_1 $ has split multiplicative reduction at $ v $.
From the representation (\ref{e2.3}), we have $ V_1^{{\cl I_v}}=\Q_p(\epsilon_p) $ {\bl where $ V_i $ is the Galois representation corresponding to $E_i, i=1,2$.}
Therefore $ \text{det}(-\text{Frob}_v \mid V_1^{{\cl I_v}})=-\epsilon_p(\text{Frob}_v)\equiv -1 $ mod $ p $.
Next, let $ E_1 $ has non-split multiplicative reduction at $ v $.
From the representation (\ref{e2.3}), we have $ V_1^{{\cl I_v}}=\Q_p(\epsilon_p \otimes \delta). $
Therefore $ \text{det}(-\text{Frob}_v \mid V_1^{{\cl I_v}})=-\epsilon_p(\text{Frob}_v) \, \delta(\text{Frob}_v) \equiv 1 $ mod $ p $, since $ \delta $ is non-trivial.

{\bl Since $ E_1[p] \cong E_2[p] $ and the value of $ \varepsilon_0(V_i) $ mod $p$ depends only on the residual representation $ E_i[p] $, we have $ \varepsilon_0(V_1) \equiv \varepsilon_0(V_2) \, \text{mod} \, p. $
Now $ E_2 $ have additive reduction implies $ V_2^{{\cl I_v}}=0 $.
Therefore $ \varepsilon_0(V_2)=\varepsilon(V_2). $
Finally from (\ref{e5.3}), (\ref{e5.4}) and the fact that $ \varepsilon_0(V_1) \equiv \varepsilon_0(V_2) \, \text{mod} \, p $, we have $ w(E_2/F_v)=1 $.}

Suppose $ \mu_p \nsubseteq F_v $.
Then $ \epsilon_p $ is non-trivial and $ \epsilon_p(\text{Frob}_v) \equiv -1 $ mod $ p. $
Now by similar arguments as above, we have $ w(E_2/F_v)=-1. $

\end{proof}

\begin{theorem}

Let $ E_1 $ and $ E_2 $ be two elliptic curves defined over $ \Q $ with good reduction at an odd prime $ p $ and $ E_1[p] \cong E_2[p]. $
Then \begin{equation} \label{e5.6}
\dfrac{w(E_1/F)}{w(E_2/F)} = (-1)^{\mid S_{E_1} \mid - \mid S_{E_2} \mid + \mid T \mid}. 
\end{equation}

\end{theorem}

\begin{proof}

When $ E_1 $ and $ E_2 $ are two semistable elliptic curves, we have already proved in Theorem \ref{t5.1} that (\ref{e5.6}) holds. 
Next assume $ E_1 $ and $ E_2 $ have additive reduction at $ v $.
Using Lemma \ref{l5.2}, we have $ w(E_1/F_v) = w(E_2/F_v). $
Also Lemma \ref{l2.3} implies that it is not possible to have a prime $ v $ such that $ E_1 $ has good reduction and $ E_2 $ has additive reduction at $ v $.
 
Let $ p \geq 5 $.
By Lemma \ref{l2.10}, it is not possible to have a prime $ v $ such that $ E_1 $ has multiplicative reduction and $ E_2 $ has additive reduction at $ v $. 
Hence assume $ p=3 $, $ E_1 $ has multiplicative reduction and $ E_2 $ has additive reduction at $ v $.
Then $ w(E_1/F_v)=-1 $ if $ E_1 $ has split multiplicative reduction at $ v $ otherwise $ w(E_1/F_v)=1. $
From Lemma \ref{l5.3}, we have $ w(E_2/F_v)=1 $ if $ \mu_p \subset F_v $ else $ w(E_2/F_v)=-1. $
Therefore if $ E_1 $ has split (respectively non-split) multiplicative reduction and $ E_2 $ has additive reduction at $ v,$ then $ w(E_1/F_v)/w(E_2/F_v) = -1 $ (respectively 1) if $ \mu_p \subset F_v $ otherwise $ w(E_1/F_v)/w(E_2/F_v) = 1 $ (respectively -1). 

From the above comparisons of $ w(E_i/F_v), i=1,2 $, we have $$ \dfrac{w(E_1/F)}{w(E_2/F)} = (-1)^{\mid S_{E_1} \mid - \mid S_{E_2} \mid + \mid T \mid}.  $$

\end{proof}

{\bl \begin{corollary}
Let $E_1$ and $E_2$ be two elliptic curves defined over $\Q$ with good reduction at $p$ and $E_1[p]\cong E_2[p]$.
Then under the assumptions of Theorem \ref{t3.2} ($ E_i $ have ordinary reduction at $ p $) or Theorem \ref{t4.11} ($ E_i $ have supersingular reduction at $ p $), we have $$ \dfrac{w(E_1/F)}{w(E_2/F)}= \dfrac{(-1)^{s_p(E_1/F)}}{(-1)^{s_p(E_2/F)}} $$
\end{corollary} }

{\rd The above corollary implies that if the $p$-parity conjecture is true for one of the congruent elliptic curves then it is also true for the other elliptic curve. }

\begin{theorem}

Let $ E_1 $ and $ E_2 $ be two elliptic curves defined over $ \Q $ with good reduction at an odd prime $ p $ and $ E_1[p] \cong E_2[p]. $
Let $ v \nmid p $ be a finite prime of $ F $ and $ w \mid v $ be a prime of $ F_{cyc} $.
Then \begin{equation} \label{e5.5} 
\dfrac{w(E_1/F_v)}{w(E_2/F_v)}=(-1)^{\sigma_{E_1}^{w} - \sigma_{E_2}^{w}}.
\end{equation}

\end{theorem}

\begin{proof}

Let $ E_1 $ and $ E_2 $ have good reduction at $ v $.
Then from Lemma \ref{l2.1}, we have $ \sigma_{E_1}^{w}= \sigma_{E_2}^{w}. $
Also $ w(E_1/F_v)=w(E_2/F_v)=1. $
Hence (\ref{e5.5}) holds.
Next assume $ E_1 $ and $ E_2 $ have additive reduction at $ v $.
From Lemmas \ref{l2.2} and \ref{l5.2}, we have $ \sigma_{E_i}^{w}=0, i=1,2 $ and $ w(E_1/F_v)=w(E_2/F_v) $ respectively.
Therefore \ref{e5.5} holds. 
Lemma \ref{l2.3} implies that it is not possible to have a prime $ v $ such that $ E_1 $ has good reduction and $ E_2 $ has additive reduction at $ v $.

Let $ p \geq 5 $.
By Lemma \ref{l2.10}, it is not possible to have a prime $ v $ such that $ E_1 $ has multiplicative reduction and $ E_2 $ has additive reduction at $ v $. 
Hence assume $ p=3, E_1 $ has multiplicative reduction and $ E_2 $ has additive reduction at $ v. $
Now $ \sigma_{E_2}^{w} = 0 $ follows from Lemma \ref{l2.2}.
Let $ \mu_p \subset F_v. $
Then $ \epsilon_p=1 $ which implies $ \sigma_{E_1}^{w}=1 $ if $ E_1 $ has split multiplicative reduction and $ \sigma_{E_1}^{w}=0 $ if $ E_1 $ has non-split multiplicative reduction at $ v $ (see Lemmas \ref{l2.7} and \ref{l2.8} ).
Therefore $ \sigma_{E_1}^{w} - \sigma_{E_2}^{w}=1 $
 if $ E_1 $ has split multiplicative reduction and $ \sigma_{E_1}^{w} - \sigma_{E_2}^{w}=0 $ if $ E_1 $ has non-split multiplicative reduction at $ v $.
 From Lemma \ref{l5.3}, we have $ w(E_2/F_v)=1. $
 Hence (\ref{e5.5}) holds.
 
Let $ \mu_p \nsubseteq F_v $.
Since $ p=3 $ and there exists a unique unramified quadratic character, we have $ \epsilon_p=\delta $ which implies $ \sigma_{E_1}^{w}=1 $ if $ E_1 $ has non-split multiplicative reduction at $ v $ (see Lemma \ref{l2.8}).
Also from Lemma \ref{l5.3}, we have $ w(E_2/F_v)=-1. $
Hence (\ref{e5.5}) holds in this case also.
 
\end{proof}

\begin{corollary}

Let $ E_1 $ and $ E_2 $ be two elliptic curves defined over $ \Q $ with good reduction at an odd prime $ p $ and $ E_1[p] \cong E_2[p]. $
Let $ v \nmid p $ be a finite prime of $ F $ and $ w \mid v $ be a prime of $ F_{cyc} $.
Then $$ \dfrac{w(E_1/F_v)}{w(E_2/F_v)}=(-1)^{\tau_{E_1}^{v} - \tau_{E_2}^{v}}. $$

\end{corollary}

\begin{proof}
From \cite[Lemma 3.2]{Sh}, we have $ \tau_{E_i}^{v} \equiv \sigma_{E_i}^{w} \enspace (\text{mod} \enspace 2), i=1,2. $
This implies $ \tau_{E_1}^{v} - \tau_{E_2}^{v} \equiv \sigma_{E_1}^{w} - \sigma_{E_2}^{w} \enspace (\text{mod} \enspace 2).   $
Now the {\cl claim} follows from (\ref{e5.5}).

\end{proof}

\end{document}